\documentclass[aap,authoryear]{imsart}
\usepackage{preamble-ims}

\begin{document}
\begin{frontmatter}

\title{Divergence-Kernel method for scores of random systems}

\begin{aug}
\author[A]{\fnms{Angxiu}~\snm{Ni} \ead[label=e1]{angxiun@uci.edu}}
\address[A]{Department of Mathematics,
University of California, Irvine \printead[presep={ ,\ }]{e1}}
\end{aug}

\begin{abstract}
We derive the divergence-kernel formula for the scores of random dynamical systems, then formally pass to the continuous-time limit of SDEs.
Our formula works for multiplicative noise systems over any period of time; it does not require hyperbolicity.
We also consider several special cases: (1) for additive noise, we give a pure kernel formula; (2) for short-time, we give a pure divergence formula; (3) we give a formula which does not involve scores of the initial distribution.
Based on the new formula, we derive a pathwise Monte-Carlo algorithm for scores, and demonstrate it on the 40-dimensional Lorenz 96 system with multiplicative noise.
\end{abstract}

\begin{keyword}[class=MSC]
\kwd[Primary ]{60H07} 
\kwd{60J60} 
\kwd{65D25} 
\kwd{65C30} 
\kwd[; secondary ]{65C05} 
\end{keyword}

\begin{keyword}
\kwd{Score function}
\kwd{Cameron-Martin-Girsanov}
\kwd{Divergence method}
\kwd{Diffusion process}
\end{keyword}

\end{frontmatter}

\maketitle

\section{Introduction}
\label{s:intro}

\subsection{Main results}

This paper rigorously derives the divergence-kernel formula for the scores of discrete-time random dynamical systems.
Then we formally pass to the continuous-time limit; here we present the continuous-time result here since it is more explicit and looks simpler.
Our result gives a fundamental tool for diffusion processes.

\begin{restatable}[formal divergence-kernel score fomula for SDE]{theorem}{goldbach}
\label{t:divker_sde}
For the SDE in $\R^M$,
\[
dx_t = F(x_t) dt + \sigma(x_t)dB_t,
\]
For any backward adapted process $\{ \alpha_t \}_{ t=0 }^T$, let $\{ \nu_t \}_{ t=0 }^T$ be the forward covector process
\begin{equation*}\begin{split}
  d \nu
  = \left( 
  (\nabla \sigma \nabla \sigma^T 
  - \nabla F^T - \alpha)\nu 
  - \nabla \div F 
  + \nabla^2 \sigma \nabla \sigma 
  + \nabla \sigma \Delta \sigma \right) dt
  - 
  (\nabla \sigma \nu^T 
  + \nabla^2 \sigma 
  + \alpha \sigma^{-1}) dB
  .
\end{split}\end{equation*}
with initial condition $ \nu_0 = \nabla \log h_0(x_0)$,
where $h_t$ is the probability density of $x_t$.
Then
\begin{equation*}\begin{split}
  \nabla \log h_T (x_T) 
  = \E{\nu_T  \middle| x_T}.
\end{split}\end{equation*}
\end{restatable}

Here $\Delta \sigma$ is the Laplacian of $\sigma$, $B$ is the Brownian motion, and the SDE is Ito.
The backward adaptation is defined in \Cref{d:backward}, and the integration of a backward adapted process $\alpha_t$ is the limit of \Cref{e:backito}.
Typically $\alpha_t \ge 0$, so the term $\alpha \nu dt$ damps the exponential growth of $\nu$; it is the portion of the divergence shifted to the probability kernel.
The main advantages of this result are: (1) it works for multiplicative noise where $\sigma$ depends on $x$; (2) the size of $\nu_T$ does not grow exponentially to $T$; (3) it does not make assumptions on dynamics such as hyperbolicity. 
The corresponding result of discrete-time in \Cref{t:divkernstep_covec} is rigorous and useful for numerical analysis.
We also consider several special cases:
\begin{itemize}
  \item For additive noise (this means $\sigma$ does not depend on $x$), \Cref{t:kernstep,t:ker_sde} give pure kernel formulas, which is simpler since it does not involve divergence terms.
  \item \Cref{t:divnstep,t:div_sde} give pure divergence formulas, which works for multiplicative noise, but the expression tends to grow exponentially fast, so it does not work for long-time.
  \item \Cref{t:divkernstep_covec_noh0,t:divker_sde_noh0} give formulas which does not involve $\nabla \log h_0$, the score of the initial distribution.
  This allows us to work with singular initial distributions.
\end{itemize}

\subsection{Literature review}

The averaged statistic of a random dynamical system is of central interest in applied sciences.
In this paper, the average is taken with respect to randomness.
We are interested in the scores $\nabla \log h_t$, where $h_t$ is the probability density at time $t$.
It is a fundamental tool for many applications in statistics and computing.
For example, \cite{MG25score,Tang2024,score_diffusion1,} proposed using scores of random systems in generative models in machine learning.
However, previously we do not have good formulas for computing such scores of random dynamical systems.

There are three basic methods for expressing and computing derivatives of the distributions of random systems: the path-perturbation method (shortened as the path method), the divergence method, and the kernel-differentiation method (shortened as the kernel method).
These methods can be used for derivatives with respect to the terminal conditions, initial conditions, and parameters of dynamics (known as the linear response).
The score is a derivative of the terminal densities. 
The relation and difference among the three basic methods can be illustrated in a one-step system, which is explained in \cite{Ni_kd}.
These three methods can run on sample paths, so they all have the potential to work in high dimensions.

The path-perturbation method computes the path-perturbation under fixed randomness (such as Brownian motion), then averages over many realizations of randomness.
This is also known as the ensemble method or the stochastic gradient method \cite{eyink2004ruelle,lucarini_linear_response_climate}.
It also includes the backpropagation method, which is the basic algorithm for machine learning.
It is good at stable systems and derivatives on initial conditions.
However, it is expensive for chaotic or unstable system; the work-around is to artificially reduce the size of the path-perturbation, such as shadowing or clipping methods \cite{Ni_nilsas,Ni_NILSS_JCP,clip_gradients2,FP25}, but they all introduce systematic errors.
Moreover, it computes the derivative of an averaged observable, but does not give the derivative of the probability density, so it is not useful for scores.

The divergence method is also known as the transfer operator method, since the perturbation of the measure transfer operator is some divergence.
It is good at unstable systems and derivatives of terminal densities.
It was used to solve the optimal response problem for expanding maps and hyperbolic systems \cite{optimal_response_23,optimalresponse2017,FP25optimalRes}.
Traditionally, for systems with contracting directions, the recursive divergence formula grows exponentially fast, so the cost of Monte-Carlo type algorithm is high for long-time.
The workaround is to use a finite-element-type algorithm, which has deterministic error rather than random sampling error, but is expensive in high dimensions \cite{Galatolo2014,Wormell2019a,Zhang2020}.
For SDEs, if we differentiate the Fokker–Planck equation, we get essentially the divergence formula; but even worse, since the extra second-order term cannot be expressed by a single path, it is even more difficult to do sampling.

The kernel-differentiation method works only for random systems. 
In SDEs, this is a direct result of the Cameron-Martin-Girsanov theorem \cite{CM44, MalliavinBook}; it is also called the likelihood ratio method or the Monte-Carlo gradient method \cite{Rubinstein1989,Reiman1989,Glynn1990}.
This method is more robust than the previous two for random dynamics, since the derivative hits the probability kernel, which is refreshed at each step \cite{sedro23}.
It was used to solve the optimal response problem for random systems \cite{DGJ25,ADF}.
It is good at taking derivative with respect to the dynamics; it naturally allows Monte-Carlo-type computation \cite{Ni_kd,PSW23,ZD21}; proofs for stationary measures were given in \cite{HaMa10,bahsoun20,GG19}.
However, it cannot handle multiplicative noise or perturbation on the diffusion coefficients.
It is also expensive when the noise is small.

Mixing two basic methods can overcome some major shortcomings.
For hyperbolic systems, the basic strategy is to partition the tangent space into stable and unstable subspace, then deal with them differently; this strategy is used in the proof of the linear response (this is the derivative with respect to the dynamics) of physical measures (this is the infinite-time-averaged statistics) \cite{Dolgopyat2004,Ruelle_diff_maps,Jiang2012,Gouezel2008,Baladi2017}.
For the pointwise expression of linear responses without exponentially growing or distributive terms, we need to use the path-perturbation on the stable, and divergence method on the unstable.
This is the fast response formula \cite{Ni_asl,TrsfOprt,fr,Ni_nilsas,far,vdivF}; it can be sampled pointwise; it is used to solve the optimal response problem in the hyperbolic setting \cite{GN25}. 
It is good at high dimensions and no-noise system.
However, it does not work when the hyperbolicity is poor \cite{Baladi2007,wormell22}, and only the unstable part involves derivatives of densities.

We can also mix the path-perturbation with the kernel methods.
The Bismut-Elworthy-Li formula \cite{Bismut84,EL94,PW19bismut} computes the derivative with respect to the initial conditions, but it does not handle $dB$-type perturbations.
The path-kernel method in \cite{dud} gives the linear response of the diffusion coefficients, where the main difficulty is that the perturbation is $dB$-type rather than $dt$-type.
It is good at systems with not too small noise and not too much unstableness, it does not require hyperbolicity, and it can handle perturbation on initial conditions.
However, it can be expensive when the noise is small and unstableness is big.
Moreover, only part of the expression is derivative on the terminal densities.

This paper should be the first example mixing the divergence and kernel methods.
Such a mixture is good at systems with not too much contraction and not too small noise; it does not require hyperbolicity.
Moreover, it naturally handles derivatives of terminal densities.

We also proposed a triad program in \cite{Ni_kd}, which requires advancing and mixing all three methods.
That might be the best solution for computing derivatives of random systems or approximate derivatives of deterministic systems.

Finally, for the particular purpose of scores of random systems, there are other results which do not fall into the above logic.
\cite{MG25score} proposes to use the Bismut formula for scores of SDE with additive noise, by transforming Skorokhod integration to Ito integration.
In contrast, our result for the additive noise in \Cref{t:ker_sde} is simpler and should be faster.
Our main result in \Cref{t:divker_sde} allows multiplicative noise, thanks to involving the divergence method.

\subsection{Structure of paper}

\Cref{s:notations} defines some basic notation and states some basic assumptions on regularities.
\Cref{s:1step} derives our results for one-step random systems.
\Cref{s:discrete} derives our results for many-step random systems.
\Cref{s:cts} formally passes to the continuous-time limit of SDEs.
In each section, we sequentially derive the kernel-differentiation formula, the divergence formula, and the divergence-kernel formula.
\Cref{s:howtouse} explains how to use the schedule process to temper the exponential growth of the covector process of $\nu$, which is caused by contraction in the system.

\Cref{s:numeric} considers numerical realizations.
\Cref{s:eg_1dim} uses \Cref{t:ker_sde,t:div_sde,t:divker_sde} to compute the score of the 1-dimensional Ornstein-Ulenback (OU) process with additive noise and multiplicative noise.
\Cref{s:eg_lorenz} uses \Cref{t:divker_sde_noh0} to compute the score of the 40-dimensional Lorenz 96 model \cite{Lorenz96} with multiplicative noise and singular initial distributions. 
These examples can not be solved by previous methods since the system is high-dimensional, has multiplicative noise, has contracting directions, has singular initial distributions, and is non-hyperbolic.

Our next step is to apply our new formulas to diffusive generative models in machine learning.
With a more powerful tool for computing scores, the previous framework in, for example, \cite{MG25score} can be greatly improved.
Then we seek to extend our method to the linear responses of diffusions, which should give us a more direct tool for optimizing the diffusion models.
We will also seek to accomplish the triad program proposed in \cite{Ni_kd}.

\section{Notations and assumptions}
\label{s:notations}

We define some geometric notations for spatial derivatives.
In this paper, we denote both vectors and covectors by column vectors in $\R^M$.
The product between a covector $\nu$ and a vector $v$ is denoted by $\cdot$, 
that is,
\begin{equation*}\begin{split}
  \nu\cdot v 
  := v\cdot \nu
  :=\nu^T v
  :=v^T \nu.
\end{split}\end{equation*}
Here $v^T$ is the transpose of matrices or (co)vectors.
Then we use $\ip{,}$ to denote the inner-product between two vectors
\begin{equation*}\begin{split}
  \ip{u,v}:=u^T v.
\end{split}\end{equation*}
Note that $\DB$ may be either a vector or a covector.

Take a finite time interval $[0,T]$, our goal is to derive a pathwise expression for the score function $\nabla \log{h_T} = \frac{\nabla h_T}{h_T}$, where 
\begin{eqnarray*}
  \nabla (\cdot):=\pp{(\cdot)}x,\quad
  \nabla_v (\cdot):=
  \nabla (\cdot) v :=
  \pp{(\cdot)}x v,
\end{eqnarray*}
Here $\nabla_YX$ denotes the (Riemann) derivative of the tensor field $X$ along the direction of $Y$.
It is convenient to think that $\nabla$ always adds a covariant component to the tensor.
The divergence of a vector field $X$ is defined by
\[ \begin{split}
  \div X:= \sum_{i=1}^M \nabla_{e_i} X,
\end{split} \]
where $e_i$ is the $i$-th unit vector of the canonical basis of $\R^M$.

For a map $g$, let $g_*$ be the Jacobian matrix, or the pushforward operator on vectors; readers may as well denote this by $\nabla g$, but $g_*$ notation works better if we want to work on manifolds in the future.
We define the covector $\div g_*$, the divergence of the Jacobian matrix $g_*$, as
\begin{equation} \begin{split} \label{e:zhao}
  \div g_* :=
  \frac{\nabla \left|  g_{*} \right|}{\left| g_{*} \right|},
\end{split} \end{equation}
where $\left|  g_{*} \right|$ is the determinant of the matrix.
See \cite{TrsfOprt,far} for why this is called a divergence; or, for the special case where $g$ is from a time-discretized SDE, see \Cref{t:div}.

For discrete-time results in \Cref{s:1step,s:discrete}, we assume that 
\begin{assumption*} 
  The probability kernel $k$, drift $f$, and diffusion $\sigma$ are $C^1$ functions.
  The density of the initial distribution $h_0$ is $C^1$ when we use it.
\end{assumption*}
For \Cref{t:divkernstep_covec_noh0}, we do not explicitly use $h_0$, so the initial distribution can be singular.
Our derivation is rigorous for the discrete-time case, but is \textit{formal} when passing to the continuous-time limit, so we do not list the technical assumptions here, and leave the rigorous proof to a later paper.
We just assume that all integrations, averages, and change of limits are legit.
We still organize our continuous-time results by lemmas and theorems, but this is just a style to organize the paper and does not imply rigorousness.

\section{One-step results} \label{s:1step}

Our ambient space is the $M$-dimensional Euclidean space $\R^M$.
Let $x_0\sim h_0$, $b_0\sim k$, the one-step dynamics is 
\begin{eqnarray*}
  x_1 = f(x_0) + \sigma (x_0) b_0.
\end{eqnarray*}
Let $h_1$ denote the density of $x_1$.
The basic ingredients of our work are the kernel-differentiation formula and the divergence formula for $\nabla \log h_1$ in one-step system.

\subsection{Kernel-differentiation formula} \label{s:ker1step}

First, we give an expression of $h_1$.
For any $C^0$ observable function $\Phi$, denote the expectation of $\Phi(x_1)$ as
\begin{equation*} \label{e:start}
  \E{\Phi(x_1)} := \int \Phi(x_1) h_1(x_1) dx_1
  = \iint \Phi(x_1) h_0(x_0) k(b_0) dx_0 db_0
\end{equation*}
In the last expression, $x_1$ is a function of the two dummy variables $x_0$ and $b_0$.
Change the variable $b_0$ to $x_1$; this involves a nontrivial $\sigma^M(x_0)$ determinant, so
\[
  \E{\Phi(x_1)}
  = \iint \Phi(x_1) h_0(x_0)p(x_0,x_1) dx_0 dx_1,
\]
where the transition function is
\begin{equation*}\begin{split}
  p(x_0,x_1) := \sigma^{-M}(x_0) 
  k\left(\frac{x_1-f(x_0)}{\sigma(x_0)}\right).
\end{split}\end{equation*}
By comparison, we have an expression of $h_1$:
\begin{equation}\label{e:h1}
  h_1(x_1) = \int h_0(x_0)p(x_0,x_1) dx_0.
\end{equation}

Take spatial derivative of $h_1$, then the probability kernel $k$ will be differentiated, and we get the kernel-differentiation formula for one-step.

\begin{lemma} [Kernel formula for one-step score] \label{t:ker1step}
\begin{equation*}
  \nabla \log h_1(x_1) 
  = \E{\frac 1{\sigma(x_0)} \nabla \log k\left(b_0\right) \middle| x_1 },
  \quad \textnormal{where} \quad 
  b_0 = \frac{x_1-f(x_0)}{\sigma(x_0)}.
\end{equation*}
\end{lemma}

\begin{proof}
Take derivative of \Cref{e:h1} with respect to $x_1$, 
\begin{equation*}
  \nabla h_1(x_1) 
  = \int h_0(x_0)\sigma^{-M-1}(x_0) 
  \nabla k\left(b_0\right) dx_0,
  \quad \textnormal{where} \quad 
  b_0 = \frac{x_1-f(x_0)}{\sigma(x_0)}.
\end{equation*}
Divide both sides by $h_1$, note that $\nabla \log k = \nabla k / k$, then we get
\begin{equation*}
  \nabla \log h_1(x_1) 
  = \frac 1{h_1(x_1)}\int
  \left[
  \sigma^{-1}(x_0) \nabla \log k\left(b_0\right) 
  \right]
  h_0(x_0) p(x_0,x_1) dx_0.
\end{equation*}
Then apply the definition of the conditional expectation.
\end{proof}

\subsection{Divergence formula} \label{s:div1step}

This section derives the divergence formula.
In \Cref{e:h1}, change the variable $x_0$ to $b_0=(x_1-f(x_0)) / \sigma(x_0)$ under fixed $x_1$.
Define
\begin{equation*}\begin{split}
  x_1 = g_{b_0}(x_0) := f(x_0) + \sigma(x_0) b_0
\end{split}\end{equation*}
Changing variable causes a nontrivial Jacobian determinant,
\[ \begin{split}
  \left| \dd {b_0}{x_0} \right|
  = \sigma^{-M}(x_0) \left| \pp {f(x_0)}{x_0} + b_0 \pp {\sigma (x_0)}{x_0} \right|
  =: \sigma^{-M}(x_0) \left| g_{b_0*} (x_0) \right|,
\end{split} \]
where $g_{b_0*}$ is the pushforward operator of $g_{b_0}$, and $ \left| g_{b_0*} \right| $ is its determinant.
So we can express $h_1$ by an integration over $b_0$,
\begin{equation} \begin{split} \label{e:h1_b}
  h_1(x_1) = \int h_0(x_0)k(b_0) \left|  g_{b_0*} (x_0) \right| ^{-1} db_0,
  \quad \textnormal{where} \quad 
  x_0 = g_{b_0}^{-1}(x_1).
\end{split} \end{equation}
For now, we assume $g_{b_0}$ is injective, since this is the case when we discretize SDEs; if it is $n$-to-1, then we need to further sum over its preimage.

\begin{lemma} [Divergence formula for one-step score] \label{t:div1step}
\begin{equation*} \begin{split}
  \nabla \log h_1 (x_1) 
  = 
  \E{ g_{b_0}^{*-1} \left(\nabla \log h_0 (x_0) - (\div g_{b_0*}) (x_0)\right) \middle| x_1 },
\end{split} \end{equation*}
where $g_{b_0}^{*-1}$ is the inverse of the pullback operator on covectors.
\end{lemma}

\begin{proof}
Take the spatial derivative.
Note that in \Cref{e:h1_b}, the dummy variable is not differentiated; the derivative hits $h_0$ and $ g_{b_0*}$ via $x_0$, since it depends on $x_1$.
For convenience, assume that we are interested in changing $x_1$ in a direction $v_1$, and $v_0:=g_{b_0*}^{-1} v_1$.
Note that here $v_1$ is independent of $x_0$.
\begin{equation*} \begin{split}
  \nabla_{v_1} h_1 (x_1) 
  = \int \left(\frac{\nabla_{v_0} h_0}{h_0}(x_0)
  - \frac{\nabla_{v_0} \left|  g_{b_0*} \right|}{\left| g_{b_0*} \right|}(x_0)
  \right)
  h_0(x_0)k(b_0) \left|  g_{b_0*} (x_0) \right| ^{-1} db_0,
\end{split} \end{equation*}
By definition of the divergence in \Cref{e:zhao},
\begin{equation*} \begin{split}
  \nabla_{v_1} h_1 (x_1) 
  = \int \left(\nabla_{v_0} \log h_0 (x_0)
  - (\div g_{b_0*}) (x_0) v_0
  \right)
  h_0(x_0)k(b_0) \left|  g_{b_0*} (x_0) \right| ^{-1} db_0,
\end{split} \end{equation*}
Then we can change back the dummy variable to $x_0$,
\begin{equation} \begin{split} \label{e:deng}
  \nabla_{v_1} h_1 (x_1) 
  = \int \left(\nabla_{v_0} \log h_0 (x_0)
  - (\div g_{b_0*}) (x_0) v_0
  \right)
  h_0(x_0) p(x_0,x_1) dx_0.
\end{split} \end{equation}
Then divide by $h_1(x_1)$ to get
\begin{equation*} \begin{split}
  \nabla_{v_1} \log h_1 (x_1) 
  = \E{ \left(\nabla \log h_0 (x_0) - (\div g_{b_0*}) (x_0) \right)v_0 \middle| x_1 }
  \\
  = \E{ \left(\nabla \log h_0 (x_0) - (\div g_{b_0*}) (x_0) \right) g_{b_0*}^{-1}v_1 \middle| x_1 }.
\end{split} \end{equation*}
Since $ v_1 $ was arbitrarily chosen, we can obtain the formula in the lemma.
\end{proof}

Readers can see \cite{TrsfOprt} for more information about \Cref{t:div1step}, including a derivation in the matrix notation, and an extension to (unstable) submanifolds.
The two one-step linear response formulas are equivalent under integration by parts.
But they perform quite differently for numerical applications and in proofs, since they put the derivative on different terms.
The difference is even sharper for multi-step and infinite-step systems.

\subsection{Divergence-kernel formula} \label{s:divker1step}

We can combine the kernel-differentiation formula and the divergence formula.
In the following lemma, $\alpha_1$ may depend on $x_1$ but not $x_0$, otherwise we would have more terms.
In other words, our ambient $\sigma$-algebra is $\sigma(x_0,x_1)$, the $\sigma$-algebra generated by $x_0$ and $x_1$ (we overload the notation $\sigma$ for both the $\sigma$-algebra and the diffusion coefficient), and $\alpha_1\in \sigma(x_1)$.

\begin{lemma} [Divergence-kernel formula for one-step score] \label{t:divker1step}
For any matrix $\alpha_1\in\R^{M\times M}$,
\begin{equation*} \begin{split}
  \nabla \log h_1 (x_1) 
  = 
  \E{
  \frac {1} {\sigma(x_0)} \alpha_1^T \nabla \log k\left(b_0\right) 
  + 
  (1-\alpha_1^T)  g_{b_0}^{*-1} \left(\nabla \log h_0  - \div g_{b_0*} \right) (x_0) 
  \middle| x_1 }.
\end{split} \end{equation*}
\end{lemma}

\begin{remark*}
  We can decompose $v_1$ into any two vectors and still have \Cref{e:yuan}.
  But to get the score function we need to decompose $v_1$ by linear operators.
\end{remark*} 

\begin{proof}
Pick any $v_1$, denote 
$ v_0 := g_{b_{0}*}^{-1} (1-\alpha_1) v_1 $, 
where $ b_{0} = \frac{x_1-f(x_{0})}{\sigma(x_{0})} $.
Since $\alpha_1$ and $v_1$ do not dependent on $x_0$ (this means fixing $x_1$ then changing $x_0$ does not change $\alpha_1$; it should not be confused with probabilistic independence),
\begin{equation*} \begin{split}
  \nabla_{v_1} h_1 (x_1) 
  = \nabla_{\alpha_1 v_1} h_1 (x_1) 
  + \nabla_{(1-\alpha_1)v_1} h_1 (x_1) 
  \\
  = \int \frac { \nabla_{\alpha_1v_1} \log k (b_{0}) } {\sigma(x_{0})} 
  h_{0} (x_{0}) p(x_{0},x_{N}) dx_{0}
  + \int \left(\nabla_{v_{0}} h_{0} - \div g_{b_{0}*} \cdot v_0  h_{0} \right)(x_{0})
  p(x_{0},x_{N}) dx_{0}.
\end{split} \end{equation*}
Then divide by $h_1$ to get
\begin{equation} \begin{split} \label{e:yuan}
  \nabla_{v_1} \log h_1 (x_1) 
  = 
  \E{
  \frac {\nabla \log k \left(b_0\right)} {\sigma(x_0)} \cdot \alpha_1 v_1
  + 
  \left(\nabla \log h_0 - \div g_{b_0*} \right) (x_0) \cdot g_{b_{0}*}^{-1} (1-\alpha_1) v_1
  \middle| x_1 }.
\end{split}\end{equation}
Then we can move all the operations away from $v_1$,
\begin{equation*}\begin{split}
  \E{
  \left(\alpha_1^T \frac {\nabla \log k \left(b_0\right)} {\sigma(x_0)}\right) \cdot  v_1
  + 
  \left(
  (I-\alpha_1^T) g_{b_{0}}^{*-1} \left(\nabla \log h_0  - \div g_{b_0*} \right) (x_0)
  \right) \cdot
  v_1
  \middle| x_1 }.
\end{split} \end{equation*}
Since $v_1$ is arbitrary, we proved the lemma.
\end{proof}

\section{Discrete-time many-steps results} \label{s:discrete}

Starting from the initial density $x_0 \sim h_0$, consider the discrete-time random dynamical system
\begin{equation} \begin{split} \label{e:discreteEq}
  x_{n+1} = f(x_n) + \sigma(x_n)b_n,
\end{split} \end{equation}
where $b_n$ is i.i.d. distributed according to the kernel density $k$.
Denote the density of $x_n$ by $h_n$.
The system runs for a total of $N$ steps.

We recursively apply the results in \Cref{s:1step} to obtain the score formulas of $h_N$, which is the density of $x_N$.
To shorten our expressions, let 
\begin{equation*}\begin{split}
  p(x_ {n\sim N} ) := p(x_n, x_{n+1}) \cdots p(x_{N-1}, x_N),
  \quad \textnormal{} \quad 
  d x_ {n\sim N} := dx_n \cdots d x_N.
\end{split}\end{equation*}

\begin{definition}
\label{d:backward}
  We say that a process $\{ \alpha_n \}_{ n=0 }^N$ is \textit{backward adapted}, or adapted to the backward filtration, if $\alpha_n$ belongs to the $\sigma$-algebra $\sigma(x_n,\ldots, x_N)$.
  We say that $\{ \beta_n \}_{ n=0 }^N$ is \textit{backward predictable} if $\beta_n$ belongs to the $\sigma$-algebra $\sigma(x_n+1,\ldots, x_N)$.
\end{definition}

\subsection{Kernel-differentiation formula} \label{s:kernstep}

For the kernel-differentiation method, there is no apparent need to do things recursively: \Cref{t:ker1step} does not depend on quantities in the previous step.
However, it does not work well when the system comes from discretizing an SDE, since $\nabla \log k$ is very large for the Gaussian kernel in one time-step.
Hence, using the method in \cite{dud}, we devise a formula in which we spread the kernel-differentiation to many steps.

For this subsection, we assume the simple case of \textit{additive noise}, that is, $\sigma$ is a constant which does not depend on $x$.
We want to compute the score of 
\begin{equation*}\begin{split}
  h_N(x_N) 
  = \int h_0(x_0)p(x_0,x_1) \cdots p(x_{N-1},x_N) dx_0\cdots dx_{N-1}.
\end{split}\end{equation*}
Compared with the result in \cite{MG25score}, which is based on the Bismut-Elworthy-Li formula, our formula is simpler since it does not involve path-perturbations.
For the more complicated case of multiplicative noise, where $\sigma$ depends on $x$, it seems that we must involve the divergence method in later parts of the paper.

\begin{theorem} [N-step kernel formula for score] \label{t:kernstep}
For the case of additive noise ($\sigma$ is independ of $x$), for any predictable or backward predictable process $\beta_n \in [0,1]$ with $\beta_N = 1$,
\begin{equation*}\begin{split}
  \nabla \log h_N(x_N) 
  = 
  \E{
  \beta_0 \nabla \log h_0(x_0)
  +
  \frac {1} {\sigma} 
  \sum _{n=0}^{N-1} 
  \left(\beta_{n+1} I - \beta_n \nabla f^T(x_n) \right)
  \nabla \log k \left(b_n\right)
  \middle|
  x_N
  } 
\end{split}\end{equation*}
\end{theorem}

\begin{remark*}
If we set $\beta_0=0$, then we do not need the derivative of $h_0$.
We can have a more fancy predictability pattern of $\beta$, for example, we may let $\beta_0\sim \beta_{\frac N2}$ be predictable, and the second half be backward predictable.
\end{remark*}

\begin{proof}
For convenience, assume that we make a perturbation $v_N$ to $x_N$, 
\begin{equation*}\begin{split}
  \nabla_{v_N}h_N(x_N) 
  = \lim_{\gamma\rightarrow 0} \frac 1{\gamma} h_N(x_N + \gamma v_N) 
\end{split}\end{equation*}
where
\begin{equation*}\begin{split}
  h_N(x_N + \gamma v_N)  
  = \int h_0(x^\gamma_0)p(x^\gamma_0,x^\gamma_1) \cdots p(x^\gamma_{N-1},x_N + \gamma v_N) dx^\gamma_0\cdots dx^\gamma_{N-1}.
\end{split}\end{equation*}

Change the variable $x^\gamma_{n}$ to $x_{n}$ according to 
\begin{equation*}
  x_{n}^\gamma = x_{n} + \gamma \beta_n v_{N}.
\end{equation*}
Since $v_N$ is constant and $\beta_n$ is either predictable or backward predictable, the Jacobian matrix of this change of variables is either upper or lower block triangular with diagonal blocks being identity matrices, so the determinant of changing variables is trivial.
Another way to look at this is that if $\beta_n$ is predictable, then we change variables sequentially from small to large $n$; if $\beta_n$ is backward predictable, then we change variables in the reverse order, then each change of variable incurs a trivial Jacobian determinant.
Hence,
\begin{equation*} \begin{split}
  h_N(x_N+\gamma v_N)  
  = \int h_0\left( x_0 + \gamma \beta_0 v_{N} \right) 
  p \left(
  x_0 + \gamma \beta_0 v_{N} ,
  x_1 + \gamma \beta_1 v_{N} \right) \\
  \cdots p \left(
  x_{N-1} + \gamma \beta_{N-1} v_{N} ,
  x_N + \gamma v_N \right) 
  dx_{0\sim N-1}.
\end{split} \end{equation*}

Then we can take derivative with respect to $\gamma$.
Recall that 
\begin{equation*}\begin{split}
  p\left(x_{n}   + \gamma \beta_{n}   v_{N} ,
         x_{n+1} + \gamma \beta_{n+1} v_{N} \right)
  =:p(x^\gamma_{n},x^\gamma_{n+1})
  = \frac 1{\sigma^M} k\left(\frac{x_{n+1}^\gamma-f(x_{n}^\gamma)} {\sigma}\right),
\end{split}\end{equation*}
denote $ b_n := \frac{x_{n+1}-f(x_{n})} {\sigma}$, then
\begin{equation*}\begin{split}
  \frac{\delta p(x^\gamma_{n},x^\gamma_{n+1})}
  {p(x_{n},x_{n+1})}
  = 
  \frac {\nabla k} k \left(b_n\right)
  \cdot\left( \beta_{n+1} - \beta_{n} \nabla f (x_n) \right)
  \frac{v_N} \sigma.
\end{split}\end{equation*}
Substitute into the derivative of $h_N$, we get
\begin{equation*}\begin{split}
  \nabla_{v_N}h_N(x_N) 
  =  
  \int \beta_0 \nabla_{v_N}  h_0(x_0)  p(x_{0\sim N}) dx_{0\sim N-1}
  \\ 
  + \int\sum _{n=0}^{N-1} 
  \left(\frac {\nabla k} k \left(b_n\right)
  \cdot \left( \beta_{n+1} - \beta_{n} \nabla f (x_n) \right) \frac{v_N} \sigma
  \right) 
  h_0(x_0) p(x_{0\sim N}) dx_{0\sim N-1}.
\end{split}\end{equation*}
Divide by $h_N(x_N)$ on both sides, we get
\begin{equation*}\begin{split}
  \nabla_{v_N} \log h_N(x_N) 
  = 
  \E{
  \beta_0 \nabla \log h_0(x_0) \cdot v_N 
  + \frac {1} { \sigma} 
  \sum _{n=0}^{N-1} 
  \frac {\nabla k} k \left(b_n\right)
  \cdot \left( \beta_{n+1} - \beta_{n} \nabla f (x_n) \right) v_N 
  \middle|
  x_N
  } 
\end{split}\end{equation*}
Since $v_N$ was chosen arbitrarily, we obtain the formula in the lemma.
\end{proof}

\subsection{Divergence formula} \label{s:divnstep}

We first derive an expanded divergence formula for the score, then a recursive formula.

\begin{theorem} [N-step divergence formula for score] \label{t:divnstep}
\begin{equation*} \begin{split}
  \nabla \log h_N (x_N) 
  = \E{ g_b^{*-N} \nabla \log h_{0}(x_0) 
  - \sum_{n=0}^{N-1} g_b^{*n-N} \div g_{b_n*}(x_{n}) \middle | x_N }.
\end{split} \end{equation*}
\end{theorem}

\begin{proof}
We use multi-variable calculus notation; the conditional expectation notation can be shorter but also a bit confusing.
Apply \Cref{e:deng} on step $N$,
\begin{equation*} \begin{split}
  \nabla h_N (x_N) 
  = \int g_{b_{N-1}}^{*-1} \left(\nabla h_{N-1} 
  - (\div g_{b_{N-1}*}) h_{N-1} \right)(x_{N-1}) p(x_{N-1},x_{N}) dx_{N-1}.
\end{split} \end{equation*}
Recursively apply this result on step $N-1$, we get
\begin{equation*} \begin{split}
  \nabla h_N (x_N) 
  = 
  \iint g_b^{*-2} \nabla h_{N-2} (x_{N-2}) p(x_{N-2},x_{N-1}) p(x_{N-1},x_{N})dx_{N-2} dx_{N-1}
  \\
  - \iint g_b^{*-2} \div g_{b_{N-2}*}(x_{N-2}) h_{N-2} (x_{N-2}) p(x_{N-2},x_{N-1}) p(x_{N-1},x_{N}) dx_{N-2} dx_{N-1}
  \\
  - \int g_{b_{N-1}}^{*-1} \div g_{b_{N-1}*}(x_{N-1}) h_{N-1} (x_{N-1}) p(x_{N-1},x_{N}) dx_{N-1}
  ,
\end{split} \end{equation*}
where $g_b^{*-2} = g_{b_{N-1}}^{*-1} g_{b_{N-2}}^{*-1}$ is the consecutive multiplication of $g_b^{*}$ evaluated at suitable places.
Keep doing this, we get
\begin{equation*} \begin{split}
  \nabla h_N (x_N) 
  = 
  \int g_b^{*-N} \nabla h_{0}(x_0) p(x_{0},x_{1})\ldots p(x_{N-1},x_{N})dx_{0} \ldots dx_{N-1}
  \\
  - \sum_{n=0}^{N-1} \int g_b^{*n-N} \div g_{b_n*}(x_{n}) h_{n} (x_{n}) p(x_{n},x_{n+1})\ldots p(x_{N-1},x_{N}) dx_n \ldots dx_{N-1}
  .
\end{split} \end{equation*}
Recall that $h_{n} = \int h_{0}(x_0) p(x_{0\sim n}) dx_{0\sim n-1}$, so
\begin{equation*} \begin{split}
  \nabla h_N (x_N) 
  = 
  \int \left( g_b^{*-N} \nabla \log h_{0}(x_0) 
  - \sum_{n=0}^{N-1} g_b^{*n-N} \div g_{b_n*}(x_{n}) \right)
  h_0(x_0) p(x_{0\sim N})dx_{0 \sim N-1}
  .
\end{split} \end{equation*}
Divide by $h_n$ and use the conditional expectation notation to get the lemma.
\end{proof}

\begin{lemma} [N-step recursive divergence formula for score] \label{t:RecurDivnstep}
Let the covector process $\nu_n$ be defined by 
\begin{equation*} \begin{split}
  \nu_0 = \nabla \log h_0 (x_0)
  ,\quad \textnormal{} \quad 
  \nu_{n+1} = g_{b_n}^{*-1} (\nu_n - \div g_{b_n*}(x_{n})).
\end{split} \end{equation*}
Then the score has the expression
\begin{equation*} \begin{split}
  \nabla \log h_N (x_N) 
  = \E{ \nu_N \middle| x_N }.
\end{split} \end{equation*}
\end{lemma}

\begin{proof}
  Check recursively that 
  $
  \nu_N
  = g_b^{*-N} \nabla \log h_{0}(x_0) 
  - \sum_{n=0}^{N-1} g_b^{*n-N} \div g_{b*}(x_{n}).
  $
\end{proof}

\subsection{Divergence-kernel formula} \label{s:divkernstep}

Even some very recent paper on score functions can not solve the case with multiplicative noise \cite{MG25score}.
We solve it by combining the kernel with the divergence method.
Then we give an example of a particular schedule, using which we can eliminate the $\nabla h_0$ in our formula.

Here the schedule $\alpha$ is similar but different to $\beta$.
\Cref{t:kernstep} uses $\beta$, which directly modulates the norm of vector process $v$.
In contrast, here $\alpha$ modulates the growth rate of the covector process $\nu$, it is deployed in the evolution equation of $\nu$.
So we require a bit different predictability for $\alpha$ and $\beta$.

\begin{lemma} [N-step backward divergence-kernel formula for score] \label{t:divkernstep}
For any backward adapted scalar process $\{ \alpha_n \}_{ n=0 }^N$, any $v_N$, let $\{ v_n \}_{ n=0 }^{N-1}$ be the backward adapted vector process
\begin{equation*}\begin{split}
v_n = (1-\alpha_{n+1}) g_{b_n*}^{-1} v_{n+1},
\end{split}\end{equation*}
with the convention $\log k (b_{-1}) :=0$, $\div g_{b_N*} :=0$, then
\begin{equation*}\begin{split}
  \nabla_{v_N} \log h_N (x_N) 
  = \E{
  \sum_{n=0}^N  \left(\frac { \alpha_n \nabla   \log k (b_{n-1}) } {\sigma(x_{n-1})} 
  -  \div g_{b_n*} (x_n) \right) \cdot v_n
  + (\nabla \log h_0(x_0)) \cdot v_0 \middle| x_N}.
\end{split}\end{equation*}
\end{lemma}

\begin{remark*}
We can choose $\alpha$ to be a matrix process, using the exact version of \Cref{t:divker1step}.
Here we use scalar $\alpha$ only for simplicity.
\end{remark*}

\begin{proof}
First, by \Cref{e:yuan}, for any scalar $\alpha_N$ and vector $v_N$ not dependent on $x_{N-1}$, 
recall that
$ b_{N-1} = \frac{x_N-f(x_{N-1})}{\sigma(x_{N-1})} $, 
$ v_{N-1} 
=g_{b*}^{-1} (1-\alpha_N)   v_N $,
then
\begin{equation*} \begin{split}
  \nabla_{v_N} h_N (x_N) 
  = \nabla_{\alpha_N v_N} h_N (x_N) 
  + \nabla_{(1-\alpha_N)v_N} h_N (x_N) 
  \\
  = \int \frac { \nabla_{\alpha_N  v_N} \log k (b_{N-1}) } {\sigma(x_{N-1})} 
  h_{N-1} (x_{N-1}) p(x_{N-1},x_{N}) dx_{N-1}
  \\
  +\int \left(\nabla_{v_{N-1}} h_{N-1} - \div g_{b_{N-1}*} \cdot v_{N-1}  h_{N-1} \right)(x_{N-1})
  p(x_{N-1},x_{N}) dx_{N-1},
\end{split} \end{equation*}
where $h_{N-1} (x_{N-1}) = \int h_{0} (x_0) p(x_{0\sim N-1}) dx_{0\sim N-2} $.

Recall that $v_{N-1}$ depends on $\alpha_N$, $v_N$, $x_N$, and $x_{N-1}$;
so $v_{N-1}$ and $\alpha_{N-1}$ do not dependent on $x_{N-2}$.
(
Note that here we used the fact that $\alpha_N$ does not depend on $x_{N-2}$, whereas the previous step uses that $\alpha_N$ does not depend on $x_{N-1}$, so we do require $\alpha_n$ to be backward adapted.
)
Also note that $ v_{N-2} = g_{b_{N-2}*}^{-1} (1-\alpha_{N-1}) v_{N-1} $, we get
\begin{equation*}\begin{split}
\nabla_{v_{N-1}} h_{N-1} (x_{N-1}) 
= 
\int \frac { \nabla_{\alpha_{N-1} v_{N-1}} \log k (b_{N-2}) } {\sigma(x_{N-2})} 
 h_{N-2} (x_{N-2}) p(x_{N-2},x_{N-1}) dx_{N-2}
\\
+\int \left(\nabla_{ v_{N-2}} h_{N-2} - \div g_{b_{N-2}*} \cdot v_{N-2} h_{N-2} \right)(x_{N-2})
  p(x_{N-2},x_{N-1}) dx_{N-2}.
\end{split}\end{equation*}

Keep doing this, we get
\begin{equation*} \begin{split}
  \nabla_{v_N} h_N (x_N) 
  = \sum_{n=1}^N  \int \frac { \nabla_{ \alpha_n v_n} \log k (b_{n-1}) } {\sigma(x_{n-1})} 
  h_{n-1} (x_{n-1}) p(x_{n-1\sim N}) dx_{n-1\sim N-1}
  \\ 
  - \sum_{n=0}^{N-1}  \int \div g_{b_n*} \cdot v_{n}  h_{n} (x_n)
  p(x_{n\sim N}) dx_{n\sim N-1}
  + \int \nabla_{v_0} h_0(x_0) p(x_{0\sim N}) dx_{0\sim N-1}.
\end{split} \end{equation*}
Then divide by $h_N (x_N)$ to get the lemma.
\end{proof}

\begin{theorem} [N-step forward divergence-kernel formula for score] \label{t:divkernstep_covec}
For any backward adapted process $\{ \alpha_n \}_{ n=0 }^N$,
let $\{ \nu_n \}_{ n=0 }^N$ be the forward covector process
\begin{equation*}\begin{split}
\nu_0 = \nabla \log h_0(x_0),
\quad \textnormal{} \quad
\nu_{n+1} = (1-\alpha_{n+1}) g_{b_n}^{*-1}
\left( \nu_{n} - \div g_{b_{n}*} (x_{n}) \right)
+ \frac {\alpha_{n+1} \nabla \log k (b_{n})} {\sigma(x_{n})} .
\end{split}\end{equation*}
Then
\begin{equation*}\begin{split}
  \nabla \log h_N (x_N) 
  = \E{\nu_N  \middle| x_N}.
\end{split}\end{equation*}
\end{theorem}

\begin{proof}
  This is the same as deriving the adjoint method for parameter derivatives, such as in \cite{Ni_asl}.
  To prove the theorem, expand everything and check terms match \Cref{t:divkernstep}.
\end{proof}

It is desirable to get rid of the $\nabla h_0$ in the expression, which is often unknown in practice.
This can be achieved by some special schedules.

\begin{theorem} [N-step forward divergence-kernel formula for score without $\nabla h_0$] \label{t:divkernstep_covec_noh0}
For any backward adapted process $\{ \alpha_n \}_{ n=0 }^N$,
let $\{ \nu_n \}_{ n=0 }^N$ be the forward covector process
\begin{equation*}\begin{split}
\nu_0 = 0,
\quad \textnormal{} \quad
\nu_{n+1} 
= g_{b_n}^{*-1}
\left( \nu_{n} - \frac nN \div g_{b_{n}*} (x_{n}) \right)
+ \frac {\nabla \log k (b_{n})} {N \sigma(x_{n})} .
\end{split}\end{equation*}
Then
\begin{equation*}\begin{split}
  \nabla \log h_N (x_N) 
  = \E{\nu_N  \middle| x_N}.
\end{split}\end{equation*}
\end{theorem}

\begin{proof}
In \Cref{t:divkernstep}, we want to have
\begin{equation*}\begin{split}
  v_n = \frac nN g_{b*}^{n-N} v_N.
\end{split}\end{equation*}
This can be achieved by 
\begin{equation*}\begin{split}
  \alpha_n = \frac 1n.
\end{split}\end{equation*}
Indeed, we can check that for this $\alpha$,
\begin{equation*}\begin{split}
  (1-\alpha_{n+1}) g_{b_n*}^{-1} v_{n+1}
  = \frac n{n+1} g_{b_n*}^{-1} \frac{n+1}N g_{b*}^{n+1-N} v_N
  = \frac nN g_{b_n*}^{n-N} v_{N}
  = v_n
\end{split}\end{equation*}
Hence, $v_0=0$, and we have
\begin{equation*}\begin{split}
  \nabla_{v_N} \log h_N (x_N) 
  = \E{
  \sum_{n=0}^N  \left(\frac { \alpha_n \nabla   \log k (b_{n-1}) } {\sigma(x_{n-1})} 
  - \div g_{b_n*} (x_n) \right) \cdot v_n
  \middle| x_N}
  \\
  = \E{
  \sum_{n=0}^N  \left(\frac { \nabla \log k (b_{n-1}) } {N \sigma(x_{n-1})} 
  - \frac nN \div g_{b_n*} (x_n) \right) \cdot g_{b*}^{n-N} v_N
  \middle| x_N}
\end{split}\end{equation*}
Move all operations away from $v_N$, we get
\begin{equation*}\begin{split}
  \nabla_{v_N} \log h_N (x_N) 
  = \E{
  \sum_{n=0}^N g_{b}^{*n-N} \left(
  \frac { \nabla \log k (b_{n-1}) } {N \sigma(x_{n-1})} 
  - \frac nN \div g_{b_n*} (x_n) 
  \right) \cdot v_N
  \middle| x_N}.
\end{split}\end{equation*}
Then remove $v_N$ and check this matches the expression in the statement.
\end{proof}

\section{Continuous-time results} \label{s:cts}

We \textit{formally} pass the discrete-time results to the continuous-time limit SDE,
\begin{equation} \begin{split} \label{e:sde}
dx_t = F(x_t) dt + \sigma(x_t)dB_t,
\end{split} \end{equation}
where $B$ denotes a standard Brownian motion.
All SDEs in this paper are in the Ito sense, so the discretized version takes the form of \Cref{e:discreteSDE}.
Denote the density of $x_t$ by $h_t$.
Let $\cF_t$ be the $\sigma$-algebra generated by $\{B\tau\}_{\tau\le t}$ and $X_0$.

Our derivation is performed on the time span divided into small segments of length $\Dt$.
Let $N$ be the total number of segments, so $N\Dt = T$.
Denote
\[ 
  \DB_n:= B_{n+1} - B_n.
\]
Denote $\alpha_n = \alpha_{n\Dt}$.
The discretized SDE is 
\begin{equation} \label{e:discreteSDE}
X_{n+1} - X_{n}
=F(X_n) \Dt + \sigma(X_n)\DB_n.
\end{equation}

This section first derives score formulas for time-discretized SDEs, then formally pass to the continuous-time limit.
We neglect high-order terms in the time-discretized result, which disappear in the continuous-time limit anyway.
Comparing the discrete-time notation in \Cref{s:discrete} with the time-discretized SDE in \Cref{e:discreteSDE}, we have
\begin{equation*}\begin{split}
  f(x) := x+ F(x)\Dt,
  \quad 
  b_n := \DB_n,
  \quad 
  k(b) := \left(2\pi \sqrt \Dt \right)^{-M/2} \exp{ -\frac {\ip{b,b}}{2\Dt}} 
  .
\end{split}\end{equation*}

\subsection{Kernel-differentiation formula} 
\label{s:kerCts}

\begin{theorem} [formal kernel score fomula for SDE in \Cref{e:sde}] \label{t:ker_sde}
Assume that $\beta_T=1$, $\beta_t$ has $C^1$ (differentiable in time) paths, whose time derivative is denoted as $\beta'$, then
\begin{equation*}\begin{split}
  \nabla \log h_T(x_T) 
  = 
  \E{
  \beta_0 \nabla \log h_0(x_0)
  + \frac 1 \sigma \int \beta_t  (\nabla F)^T (x_t) - \beta_t' dB_t
  \middle|
  x_T
  } 
\end{split}\end{equation*}
\end{theorem}

\begin{proof}
The first main term in \Cref{t:kernstep} becomes
\begin{equation}\begin{split} \label{e:nablak}
  \nabla \log k \left(\DB_n\right)
  = - \frac {\DB_n}{\Dt}.
\end{split}\end{equation}
Note that here $\DB$ is a covector.
Denote $\beta_n:= \beta_{nT/N}$, then the other main term becomes
\begin{equation*}\begin{split}
  \beta_{n+1} I - \beta_n \nabla f(x_n) 
  = \beta_{n+1} I - \beta_n (I + \Dt \nabla F(x_n)) 
  \\
  = (\beta_{n} + \beta_n'\Dt )I - \beta_n (I + \Dt \nabla F(x_n)) + O(\Dt^2)
  \\
  = (\beta_n' I - \beta_n  \nabla F(x_n)) \Dt + O(\Dt^2)
\end{split}\end{equation*}

Substitute into \Cref{t:kernstep}, we get 
\begin{equation*}\begin{split}
  \nabla \log h_N(x_N) 
  = 
  \E{
  \beta_0 \nabla \log h_0(x_0)
  +
  \frac 1 \sigma
  \sum _{n=0}^{N-1} 
  - 
  \left((\beta_n' I - \beta_n  \nabla F^T(x_n)) \Dt + O(\Dt^2) \right)\frac {\DB_n}{\Dt}
  \middle|
  x_N
  } 
\end{split}\end{equation*}
Neglect the sum of $O(\DB \Dt)$ terms,
\begin{equation*}\begin{split}
  \nabla \log h_N(x_N) 
  = 
  \E{
  \beta_0 \nabla \log h_0(x_0)
  - 
  \frac 1 \sigma
  \sum _{n=0}^{N-1} 
  (\beta_n' I - \beta_n  \nabla F^T (x_n) ) \DB_n
  \middle|
  x_N
  } 
\end{split}\end{equation*}
Formally pass to the limit $\Dt \rightarrow 0$ to get the theorem.
\end{proof}

\subsection{Divergence formula}
\label{s:divCts}

The main terms in \Cref{t:divkernstep} are 
$\div g_{b_n*}(x_{n})$ and $ g_{b_n}^{*-1} $.
For time-discretized SDEs, recall
\begin{equation*}\begin{split}
  g_{b_n}(x_n) 
  = x_n + F(x_n)\Dt + \sigma(x_n) \DB_n,
  \\
  g_{b_n *}(x_n) 
  = I + \nabla F(x_n)\Dt +  \DB_n \nabla \sigma^T(x_n).
\end{split}\end{equation*}
Here $\nabla F$ is a matrix whose $i$-th row $j$-th column is 
\begin{equation*}\begin{split}
  \left[\nabla F \right]_{ij}
  = \pp {F^i} {x^j}
  =: \partial_j F^i.
\end{split}\end{equation*}
And $\nabla \sigma^T$ is a row vector, so $\DB \nabla \sigma^T$ is a matrix whose $i$-th row $j$-th column is 
\begin{equation*}\begin{split}
  \left[\DB \nabla \sigma^T \right]_{ij}
  = \pp {\sigma} {x^j} \DB^i
  =: \partial_j \sigma \DB^i.
\end{split}\end{equation*}
Ignoring terms which go to zero in the continuous-time limit, we can derive more explicit expressions of the main terms.
This subsection first derives formal expression of $\div g_{b*}$ and $g_{b_n}^{*-1}$, then obtain the divergence formula for the scores of SDEs.

\begin{lemma} [formal approximation of $\div g_{b*}$] \label{t:div}
Ignoring terms which go to zero in the continuous-time limit,
\begin{equation*}\begin{split}
  \div g_{b_n *} (x_n)
  \approx 
  \nabla \div F(x_n) \Dt 
  + \nabla^2 \sigma(x_n) \DB
  - \nabla^2 \sigma \nabla \sigma (x_n) \Dt.
\end{split}\end{equation*}
Here $ \nabla^2 \sigma \DB = \sum_i \nabla \partial_i \sigma(x_n) \DB_n^i$,
$\nabla^2 \sigma \nabla \sigma  \Dt = \sum_i \nabla \partial_i \sigma(x_n) \partial_i \sigma(x_n) \Dt$.
\end{lemma}

\begin{proof}
Recall our definition in \Cref{e:zhao},
\begin{equation*}\begin{split}
  \div g_{b_n *}(x_n) 
  := \frac{\nabla \left| g_{b_n*} \right|}{\left| g_{b_n*} \right|}(x_n).
\end{split}\end{equation*}
We temporarily neglect the subscript $n$ and $(x_n)$,
\begin{equation*}\begin{split}
  \left| g_{b*} \right|
  = \det (I + \nabla F(x)\Dt +  \DB \nabla \sigma^T(x))
  \\
  = \det \begin{bmatrix} 
    \ddots& \\
          & 1 + \partial_i F^i \Dt + \partial_i \sigma \DB^i & \cdots 
          & \partial_j F^i \Dt + \partial_j \sigma \DB^i   
          \\
          & \vdots & \ddots & \vdots  
          \\
          & \partial_i F^j \Dt + \partial_i \sigma \DB^j & \cdots 
          & 1 + \partial_j F^j \Dt + \partial_j \sigma \DB^j
          \\
          & & & & \ddots
  \end{bmatrix},
\end{split}\end{equation*}
where $i, j$ labels different directions in $\R^M$.

We compute the determinant by expansion, or the Leibniz formula, while ignoring high-order terms.
All the off-diagonal terms are $O(\DB)$, so if we choose three or more off-diagonal terms, then it is an $O(\DB^3)$ term, which can be ignored.
So we can only choose two off-diagonal terms; after ignoring high-order terms, we are left with terms like
\begin{equation*}\begin{split}
  (\partial_j F^i \Dt + \partial_j \sigma \DB^i)
  (\partial_i F^j \Dt + \partial_i \sigma \DB^j)
  \approx
  \partial_j \sigma \DB^i \partial_i \sigma \DB^j.
\end{split}\end{equation*}
Since $B^i$ and $B^j$ are independent Brownian motions, above cross terms with $i\neq j$ can be ignored.
So the determinant is approximately only a product of the diagonal entries.
Again, ignoring high-order terms which go to zero in the continuous-time limit,
\begin{equation*}\begin{split}
  \left| g_{b*} \right|
  \approx
  1 + \sum_{i=1}^M \left(\partial_i F^i \Dt + \partial_i \sigma \DB^i\right)
  =
  1 +  \div F \Dt 
  + \nabla \sigma \cdot \DB.
\end{split}\end{equation*}

Substitute into the divergence, we get
\begin{equation*}\begin{split}
  \div g_{b*}
  = \frac{\nabla \left|  g_{b*} \right|}{\left|  g_{b*} \right|} 
  = \frac{\nabla \div F \Dt + \nabla (\nabla \sigma\cdot \DB) }
  {1 +  \div F \Dt + \nabla \sigma\cdot \DB } .
\end{split}\end{equation*}
Here $ \nabla (\nabla \sigma \cdot \DB) $ is a vector whose $i$-th entry is 
$
\partial_i \sum_j \partial_j \sigma \DB^j
$, so 
\begin{equation*}\begin{split}
   \nabla (\nabla \sigma \cdot \DB) = \nabla^2 \sigma \DB,
\end{split}\end{equation*}
where the Hessian $\nabla^2 \sigma$  is a symmetric $M\times M$ matrix with $
[\nabla^2 \sigma]_{ij} = \partial_i \partial_j \sigma.
$

Recall that $1/(1+\eps) = 1 - \eps +O(\eps^2)$, ignoring high-order terms, we get
\begin{equation*}\begin{split}
  \div g_{b*}
  \approx 
  \left(
  \nabla \div F \Dt + \nabla^2 \sigma \DB
  \right)
  \left(
  {1 -  \div F \Dt - \nabla \sigma\cdot \DB}
  \right)
  \\
  \approx \nabla \div F \Dt 
  + \nabla^2 \sigma \DB
  - (\nabla \sigma \cdot \DB) \nabla^2 \sigma \DB
\end{split}\end{equation*}
For the last term, ignore cross terms ($\DB^i\DB^j$ with $i\neq j$) and apply $\DB^i \DB^i \approx \Dt$, 
\begin{equation*}\begin{split}
  (\nabla\sigma \cdot \DB) \nabla^2 \sigma \DB
  = \sum_i \sum_j \partial_j \sigma \DB^j \nabla \partial_i \sigma \DB^i
  \approx 
  \sum_i \nabla \partial_i \sigma \partial_i \sigma \Dt
  = \nabla^2 \sigma \nabla \sigma  \Dt
  .
\end{split}\end{equation*}

\end{proof}

\begin{lemma} [formal approximation of $g_{b_n}^{*-1}$] \label{t:ginv}
Ignoring terms which go to zero in the continuous-time limit,
\begin{equation*}\begin{split}
  g_{b_n}^{*-1}
  \approx 
  I- \nabla F^T(x_n) \Dt - \nabla \sigma(x_n) \DB_n^T 
  + 
  \nabla \sigma \nabla \sigma^T (x_n)\Dt.
\end{split}\end{equation*}
\end{lemma}

\begin{proof}
In $\R^M$, under our consensus that both vectors and covectors are column vectors, the pullback operator on covectors is just the transpose of the pushforward operator, so
\begin{equation*}\begin{split}
  g_{b_n}^*(x_n) 
  = I + \nabla F^T (x_n) \Dt + \nabla \sigma(x_n) \DB_n^T.
\end{split}\end{equation*} 
For a small matrix $\eps$, $(I+\eps)^{-1} = I - \eps + \eps^2 +O(\eps^3)$, so
\begin{equation*}\begin{split}
  g_{b_n}^{*-1}
  \approx 
  I- \nabla F^T \Dt - \nabla \sigma(x_n) \DB_n^T + \left(
  \nabla F^T (x_n) \Dt + \nabla \sigma(x_n) \DB_n^T 
  \right)^2
  \\
  \approx 
  I- \nabla F^T \Dt - \nabla \sigma(x_n) \DB_n^T 
  + \nabla \sigma(x_n) \DB_n^T \nabla \sigma(x_n) \DB_n^T ,
\end{split}\end{equation*} 
where $
  \left[ \nabla \sigma \DB^T \right]_{ij} = \partial_i \sigma \DB^j
$.
For the last matrix, we ignore cross terms to get
\begin{equation*}\begin{split}
  \left[\nabla \sigma \DB^T \nabla \sigma \DB^T \right]_{ij}
  = \sum_k \left[\nabla \sigma \DB^T\right]_{ik}
  \left[\nabla \sigma \DB^T \right]_{kj}
  \\
  = \sum_k \partial_i \sigma \DB^k
  \partial_k \sigma \DB^j
  \approx
  \partial_i \sigma \partial_j \sigma \Dt
\end{split}\end{equation*}
Hence,  $
\nabla \sigma \DB^T \nabla \sigma \DB^T
\approx 
\nabla \sigma \nabla \sigma^T \Dt.
$
\end{proof}

Let $\Delta \sigma$ denote the Laplacian of the scalar function $\sigma$ (note that we overload the notation $\Delta$ for both Laplacian and small steps), we have

\begin{theorem} [formal divergence score fomula for SDE in \Cref{e:sde}] \label{t:div_sde}
Let $\{ \nu_t \}_{ t=0 }^T$ be the forward covector process
\begin{equation*}\begin{split}
  d \nu
  = \left( 
  (\nabla \sigma \nabla \sigma^T 
  - \nabla F^T )\nu 
  - \nabla \div F 
  + \nabla^2 \sigma \nabla \sigma 
  + \nabla \sigma \Delta \sigma \right) dt
  - 
  (\nabla \sigma \nu^T 
  + \nabla^2 \sigma ) dB
  .
\end{split}\end{equation*}
with initial condition $ \nu_0 = \nabla \log h_0(x_0)$,
then
\begin{equation*}\begin{split}
  \nabla \log h_T (x_T) 
  = \E{\nu_T  \middle| x_T}.
\end{split}\end{equation*}
\end{theorem}

\begin{proof}
For convenience, we collect the results results in \Cref{t:RecurDivnstep,t:div,t:ginv},
\begin{equation*}\begin{split}
  \nu_{n+1} = g_{b_n}^{*-1} (\nu_n - \div g_{b_n*}(x_{n})),
  \quad \textnormal{where} \quad 
  \\
  \div g_{b_n *} (x_n)
  \approx 
  \nabla \div F(x_n) \Dt 
  + \nabla^2 \sigma(x_n) \DB
  - \nabla^2 \sigma \nabla \sigma (x_n) \Dt.
  \\
  g_{b_n}^{*-1}
  \approx 
  I- \nabla F^T(x_n) \Dt - \nabla \sigma(x_n) \DB_n^T 
  + 
  \nabla \sigma \nabla \sigma^T (x_n)\Dt.
\end{split}\end{equation*}


By substitution, we get
\begin{equation*}\begin{split}
  \nu_{n+1} 
  \approx
  (I- \nabla F^T \Dt - \nabla \sigma \DB^T 
  + 
  \nabla \sigma \nabla \sigma^T \Dt)
  (\nu - \nabla \div F \Dt 
  - \nabla^2 \sigma \DB
  + \nabla^2 \sigma \nabla \sigma \Dt )
\end{split}\end{equation*}

Denote $\Delta \nu_n:= \nu_{n+1} - \nu_n$, then
\begin{equation*}\begin{split}
  \Delta \nu_n 
  \approx
  - \nabla F^T \nu \Dt - \nabla \sigma \DB^T \nu 
  + 
  \nabla \sigma \nabla \sigma^T \nu \Dt
  - \nabla \div F \Dt 
  \\
  - \nabla^2 \sigma \DB
  + \nabla^2 \sigma \nabla \sigma \Dt
  + \nabla \sigma \DB^T \nabla^2 \sigma \DB.
\end{split}\end{equation*}
For the last term, ignoring the cross terms, we get
\begin{equation*}\begin{split}
  \DB^T \nabla^2 \sigma \DB 
  = \sum_{i,j} \DB^i \partial_i\partial_j \sigma \DB^j
  \approx \sum_i \partial_i^2 \sigma \Dt
  =: \Delta \sigma \Dt ,
\end{split}\end{equation*}
Also note that $ \DB^T \nu = \nu^T \DB $ is a scalar, so
\begin{equation*}\begin{split}
  \Delta \nu_n 
  \approx
  \left( 
  (\nabla \sigma \nabla \sigma^T 
  - \nabla F^T )\nu 
  - \nabla \div F 
  + \nabla^2 \sigma \nabla \sigma 
  + \nabla \sigma \Delta \sigma \right) \Dt
  - \nabla \sigma \nu^T \DB
  - \nabla^2 \sigma \DB
  .
\end{split}\end{equation*}
Then we let $\Dt\rightarrow 0$ and formally pass to the limit to get the theorem.
\end{proof}

\subsection{Divergence-kernel formula} 
\label{s:divkerCts}

We formally pass \Cref{t:divkernstep_covec} to the continuous-time limit.
Then we formally pass \Cref{t:divkernstep_covec_noh0} to the continuous-time limit, which is an example of using a particular schedule to remove $\nabla \log h_0$ from the expression; there are also other schedules achieving this goal.

\goldbach*


\begin{remark*}
  If the system has decay of correlations and $T$ is large, we can arbitrarily set the initial condition, and the effect would be small.
\end{remark*}

\begin{proof}
In \Cref{t:divkernstep_covec}, we pass  $\alpha_n$  to $\alpha_n \Dt$, then the covector process $\{ \nu_n \}_{ n=0 }^N$ solves
\begin{equation*}\begin{split}
\nu_{n+1} 
= 
(1-\alpha_{n+1}\Dt) g_{b_n}^{*-1}
\left( \nu_{n} -  \div g_{b_{n}*} (x_{n}) \right)
+ \frac { \alpha_{n+1}\Dt \nabla \log k (\DB_{n}) } {\sigma(x_{n})} 
\end{split}\end{equation*}
Since $ \nabla \log k \left(\DB\right) = - \DB/ \Dt$, the $\Dt$ terms cancel,
\begin{equation*}\begin{split}
\nu_{n+1} 
= (1-\alpha_{n+1}\Dt) g_{b_n}^{*-1}
\left( \nu_{n} -  \div g_{b_{n}*} (x_{n}) \right)
- \frac { \alpha_{n+1} \DB_{n} } {\sigma(x_{n})} .
\end{split}\end{equation*}
By substitution, we get
\begin{equation*}\begin{split}
  \nu_{n+1} 
  \approx
  (1-\alpha_{n+1}\Dt) 
  (I- \nabla F^T \Dt - \nabla \sigma \DB^T 
  + 
  \nabla \sigma \nabla \sigma^T \Dt)
  \\
  (\nu - \nabla \div F \Dt 
  - \nabla^2 \sigma \DB
  + \nabla^2 \sigma \nabla \sigma \Dt )
  - \frac { \alpha_{n+1} \DB} {\sigma} .
\end{split}\end{equation*}
By the same arguments as in \Cref{t:div_sde},
\begin{equation}\begin{split}
\label{e:backito}
  \Delta \nu_n 
  \approx
  \left( 
  (\nabla \sigma \nabla \sigma^T 
  - \nabla F^T -\alpha_{n+1} )\nu 
  - \nabla \div F 
  + \nabla^2 \sigma \nabla \sigma 
  + \nabla \sigma \Delta \sigma 
  \right) \Dt
  \\
  - (\nabla \sigma \nu^T 
  + \nabla^2 \sigma 
  + \alpha_{n+1} \sigma^{-1} ) \DB
  .
\end{split}\end{equation}
Here $\alpha_{n+1}$ is evaluated at step $n+1$; all other terms are evaluated at step $n$ and location $x_n$. 
Then we let $\Dt\rightarrow 0$ and formally pass to the limit to obtain the theorem.
\end{proof}

We can eliminate $\nabla \log h_0$ using certain schedules.
The following theorem is an example; it is parallel to the Bismut-Elworthy-Li (BEL) formula, which considers derivatives with respect to singular initial conditions.
Both the BEL and the current formula are good for theoretical purposes, since they reduce some regularity requirements.
For numerical computations, they both require that $T$ be neither too small nor too large.
For small $T$, the last term in the expression below is large, since too much is shifted from the divergence to the kernel in a short time, so $\nu$ is large; for large $T$, the lack of damping on $\nu$ makes it too large.
And large $\nu$ requires more samples.
For large $T$, we can combine two schedules: for large $t$, we aim to temper the growth of the score; for small $t$, we aim to eliminate $\nabla \log h_0$.

\begin{theorem} [formal divergence-kernel score fomula for SDE, without $\nabla h_0$] \label{t:divker_sde_noh0}
Let $\{ \nu_t \}_{ t=0 }^T$ be the forward covector process
\begin{equation*}\begin{split}
  d \nu
  = \left( 
  (\nabla \sigma \nabla \sigma^T 
  - \nabla F^T)\nu 
  + \frac tT \left(
  \nabla^2 \sigma \nabla \sigma 
  + \nabla \sigma \Delta \sigma 
  - \nabla \div F 
  \right) 
  \right) dt
  - 
  \left(\nabla \sigma \nu^T 
  + \frac tT \nabla^2 \sigma 
  + \frac 1{T \sigma}\right) dB
  .
\end{split}\end{equation*}
with initial condition $ \nu_0 = 0$,
then
\begin{equation*}\begin{split}
  \nabla \log h_T (x_T) 
  = \E{\nu_T  \middle| x_T}.
\end{split}\end{equation*}
\end{theorem}

\begin{remark*}
  Let $\nu'_t$ be the solution of the \Cref{t:divker_sde} with the particular schedule
  $
  \alpha_t = 1/t.
  $
  Let $\nu_t$ be the solution of this theorem, then we can formally check 
  \begin{equation*}\begin{split}
    \frac tT \nu_t' = \nu_t.
  \end{split}\end{equation*}
\end{remark*}

\begin{proof}
In \Cref{t:divkernstep_covec_noh0}, note that $N =T/\Dt$, $t_n=n\Dt$, so 
\begin{equation*}\begin{split}
\nu_{n+1} 
= g_{b_n}^{*-1}
\left( \nu_{n} - \frac {t_n}T \div g_{b_{n}*} (x_{n}) \right)
+ \frac {\nabla \log k (b_{n})\Dt} {T \sigma(x_{n})} .
\end{split}\end{equation*}
By substituting \Cref{t:div,t:ginv} and \Cref{e:nablak}, and ignoring subscript $n$, we get
\begin{equation*}\begin{split}
  \nu_{n+1} 
  \approx
  (I- \nabla F^T \Dt - \nabla \sigma \DB^T + \nabla \sigma \nabla \sigma^T \Dt)
  \\
  \left(\nu - \frac tT \nabla \div F \Dt 
  - \frac tT \nabla^2 \sigma \DB
  + \frac tT \nabla^2 \sigma \nabla \sigma \Dt \right)
  - \frac { \DB} {T \sigma} .
\end{split}\end{equation*}
By the same arguments as in \Cref{t:div_sde},
\begin{equation*}\begin{split}
  \Delta \nu_n 
  \approx
  \left( 
  (\nabla \sigma \nabla \sigma^T 
  - \nabla F^T)\nu 
  +\frac tT (- \nabla \div F 
  + \nabla^2 \sigma \nabla \sigma 
  + \nabla \sigma \Delta \sigma )\right) \Dt
  \\
  - \left(\nabla \sigma \nu^T 
  + \frac tT \nabla^2 \sigma 
  + \frac 1 {T \sigma} \right) \DB
  .
\end{split}\end{equation*}
Then we let $\Dt\rightarrow 0$ and formally pass to the limit to get the theorem.
\end{proof}

\subsection{How to use in cases with contracting directions, and shortcomings}
\label{s:howtouse}

We explain how to select the schedule $\alpha_t$.
Note that any choice of $\alpha_t$ yields the same score, but $\nu_t$ depends on $\alpha_t$, so different choices of $\alpha_t$ significantly affect the numerical performances.
The main goal is to obtain a small $\nu_N$, so we need fewer samples to compute the expectation in the score formula.
It is easy to see that we do not want $\alpha$ too large, which would produce a large inhomogeneous term in the SDE of $\nu$ in \Cref{t:divker_sde}.
But we also want $\alpha$ to be large enough, so that the gradient growth of the score is \textit{tempered} by shifting the divergence formula to kernel-differentiation, which we explain below.

For deterministic systems, the divergence method works poorly when the system has \textit{contracting} directions, that is, when the consecutive multiplication of $I-\nabla F^T \Dt$ grows exponentially fast.
This is different from the path-perturbation method, which involves consecutive multiplication of $I+\nabla F \Dt$, so we worry about expanding directions.
More specifically, let $D^t$ be the operator such that $D^t\nu_0$ solves the homogeneous ODE,
\begin{equation*}\begin{split}
  d \nu
  = \left( \nabla \sigma \nabla \sigma^T - \nabla F^T \right) \nu dt
  - \nabla \sigma \nu^T dB
  .
\end{split}\end{equation*}
from the initial condition $v_0$.
Then a safe choice is to select the constant $\alpha$ with
\begin{equation*}\begin{split}
  \alpha > \frac 1T \log |D^T|, 
\end{split}\end{equation*}
where $|D^T|$ is the operator norm.
For stationary measures, we let $T\rightarrow\infty$.
This ensures that $\nu$ does not grow exponentially when $T$ increases.
In practice, we may also find a large enough $\alpha$ by checking that the norm of $\nu_N$ does not explode.

Should we care very much about the numerical speed, we should design $\alpha_t$ more carefully.
Note that $\alpha_t$ can be a process, not just a fixed number.
If the system has regions with large contraction, then $\alpha_t$ should be large there to reduce the growth of $\nu_t$.
If the system has regions with small noise, then $\alpha_t$ should be small, since the $1/\sigma$ factor is already large.
The divergence-kernel formula would yield the same derivative, but an adapted $\alpha_t$ gives a smaller $\nu_N$, so we need fewer samples to take the conditional average.
In this paper, we only test the simple case where $\alpha_t$ is a fixed number.

Our divergence-kernel method combines the divergence and kernel-differentiation methods, but does not involve the path-perturbation method.
Hence, it does not work in highly contracting systems with little to no noise, the case where the path-perturbation method is good at.
Since the system is contracting, we need a large $\alpha$ to tame the score explosion, but the noise is small, so the $\frac 1 \sigma$ factor is large, so the term $\alpha/\sigma$ is large, so $\nu$ is large.
This requires many samples to compute the expectation, drastically raising the computational cost.
The path method is not a panacea either; it does not work when there are expanding directions, and it does not give the derivative of density.
We may need to use all three methods together for an ultimate solution, as proposed in \cite{Ni_kd}.

\section{Numerical examples}
\label{s:numeric}

We demonstrate our results in a few examples.
In the more important applications, such as diffusions models and linear responses, we typically do not need to really compute the score at a specific point, but we are satisfied with a \textit{sample} of the score using our new formulas.
For those more significant cases, additional knowledge is required, which would deviate from the main focus of this paper, so we leave them to future papers.

For this paper, we want to verify that our new formulas can compute scores at a prescribed point, by averaging our expression over sample paths.
For this purpose, there are some obstructions beyond the scope of this paper:
(1) generating paths with prescribed terminal conditions;
(2) compute $\nabla \log h_T(x_T)$ in another known way to compare with.
To avoid these issues, we design two specific examples, the first is one-dimensional, and the second is high-dimensional but we care about some integration of $\nabla h_T$.

\subsection{1-dimensional OU process with multiplicative noise}
\label{s:eg_1dim}

We use \Cref{t:ker_sde,t:div_sde,t:divker_sde} to compute the score of the one-dimensional Ornstein-Ulenback (OU) process with additive noise and multiplicative noise.
\begin{equation*}\begin{split}
  dx = - x^3 dt + \sigma_i dB,
  \quad \textnormal{where} \quad 
  \sigma_1(x):\equiv 1,
  \quad
  \sigma_2(x):=0.5 + \exp{-x ^2}
\end{split}\end{equation*}
The initial distribution is $x_0\sim N(0,1)$.
Set $\Dt=0.002$.

We are only interested in terminal conditions $x_T$ in the interval $[-1.8,1.8]$, which is further partitioned into $9$ sub-intervals.
We simulate many sample paths, then count the number of paths whose $x_T$ fall into a specific sub-interval, to obtain the empirical distribution of $x_T$ on sub-intervals.
Then we compute the score by averaging our new formulas over the paths, and see if our results reflect the trend of $\log h_T$ from the empirical distribution.

We apply the pure kernel method in \Cref{t:ker_sde} on the SDE with $\sigma_1$.
Set $\beta_t = t/T$, so we do not need $\nabla \log h_0$, and 
\begin{equation*}\begin{split}
  \nabla \log h_T(x_T) 
  = 
  \E{ \frac 1 {T \sigma_1} \int t  (\nabla F)^T (x_t) - 1 dB_t
  \middle| x_T } 
\end{split}\end{equation*}
Note that $\sigma_1$ is constant since the pure kernel method does not work for multiplicative noise.
The result is in \Cref{f:1d_ker}; the scores we compute correctly reflect the trend of $\log h_T$.

\begin{figure}[ht] \centering
  \includegraphics[width=0.45\textwidth]{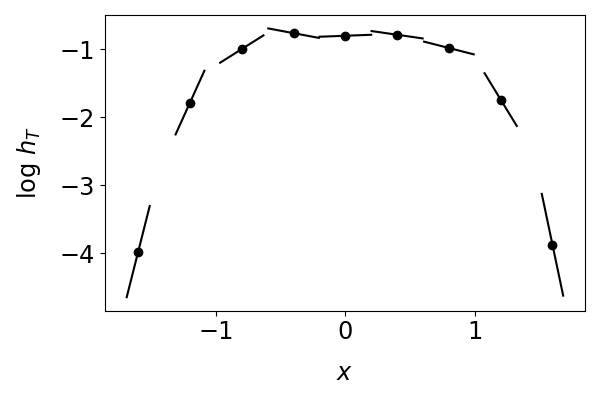}
  \caption{
  Kernel-differentiation method for 1-dimensional SDE with additive noise, $T=3$.
  $\log h_T$ and the score $\nabla \log h_T$ for different $x_T$.
  The dots are $\log h_T$.
  Each short line is a score computed by the kernel-differentiation algorithm, averaged on 10 paths whose terminal conditions $x_T$ fall into the same sub-interval.
  }
  \label{f:1d_ker}
\end{figure}

Then we apply the pure divergence method in \Cref{t:div_sde} on $\sigma_2$, since the system is $1$-dimensional, the SDE of $\nu$ becomes
\begin{equation*}\begin{split}
  d \nu
  = \left( 
  (\sigma'^2 - F' )\nu 
  - F'' + 2 \sigma'' \sigma'
  \right) dt
  - 
  (\sigma' \nu +  \sigma'' ) dB ,
  \quad \textnormal{} \quad 
  \nu_0 = \nabla \log h_0(x_0)= -x_0.
\end{split}\end{equation*}
The results are in \Cref{f:1d_div}; we can see that the pure divergence method works well for a short time period $T=0.1$, but it does not work well for a long time period $T=3$.

\begin{figure}[ht] \centering
  \includegraphics[width=0.45\textwidth]{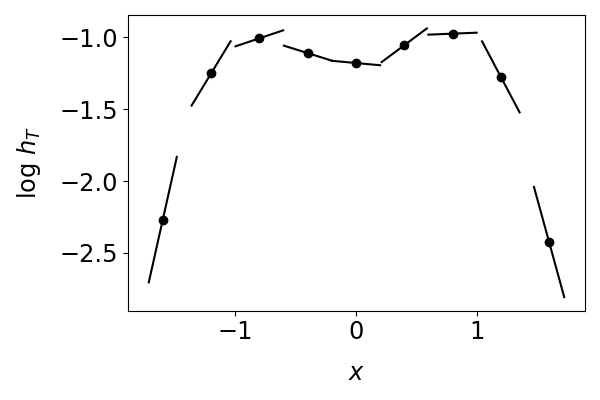} \hfill
  \includegraphics[width=0.45\textwidth]{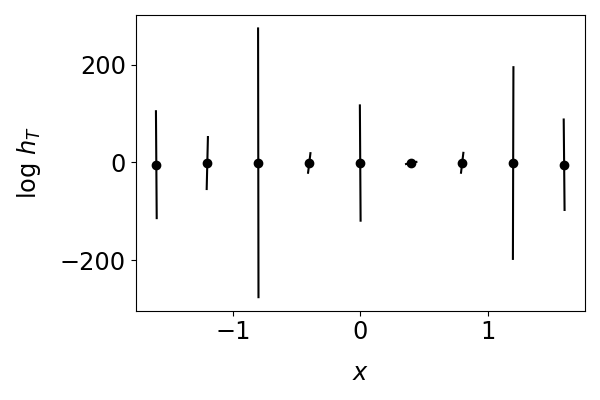}
  \caption{
  Divergence method for 1-dimensional SDE with multiplicative noise.
  Left: $T=0.1$. 
  Right: $T=3$.
  Other settings are the same as \Cref{f:1d_ker}.
  }
  \label{f:1d_div}
\end{figure}

Finally, we apply the divergence-kernel method in \Cref{t:divker_sde} on $\sigma_2$ with a large $T$.
Set $\alpha \equiv 10$, the SDE of $\nu$ becomes
\begin{equation*}\begin{split}
  d \nu
  = \left( 
  (\sigma'^2 - F' - \alpha )\nu 
  - F'' + 2 \sigma'' \sigma'
  \right) dt
  - 
  (\sigma' \nu +  \sigma'' + \alpha/\sigma ) dB
  .
\end{split}\end{equation*}
The results of the divergence-kernel method match correctly $\nabla \log h_T$ computed from the empirical distribution of $x_T$.
The results are in \Cref{f:1d_divker}.

\begin{figure}[ht] \centering
  \includegraphics[width=0.45\textwidth]{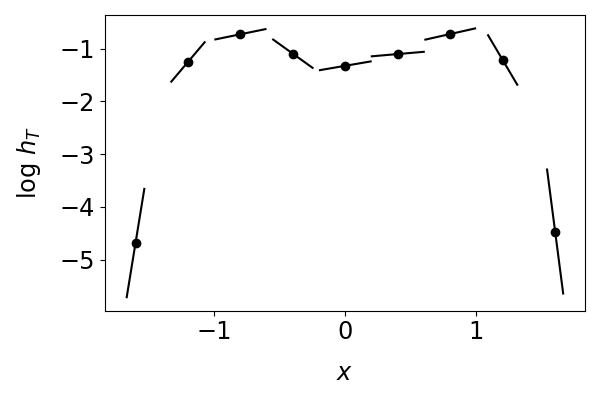}
  \caption{
  Divergence-kernel method for 1-dimensional SDE with multiplicative noise.
  $T=3$, $\alpha=10$.
  Other settings are the same as \Cref{f:1d_ker}.
  }
  \label{f:1d_divker}
\end{figure}

\subsection{40-dimensional Lorenz 96 system with multiplicative noise}
\label{s:eg_lorenz}

We use \Cref{t:divker_sde_noh0} to compute the score of the Lorenz 96 model \cite{Lorenz96} with multiplicative noise in $\R^M$. 
The dimension of the system is $M=40$.
The SDE is 
\begin{eqnarray*}
d x^i
= \left( \left(x^{i+1}-x^{i-2}\right) x^{i-1} - x^i + 8 - 0.01 (x^i)^2 \right) dt + (2+\sigma(x)) dB^i
\quad \textnormal{where}
\\
\sigma(x) = \exp{- |x|^2/2};
\quad
i=1, \ldots, M;
\quad
x_0 = [1,\ldots,1].
\end{eqnarray*}
Here $i$ labels different directions in $\R^M $, and it is assumed that $x^{-1}=x^{M-1}, x^0=x^M$ and $x^{M+1}=x^1$. 
We added noise and the $ - 0.01 (x^i)^2$ term, which prevents the noise from carrying us to infinitely far away.
The terms in \Cref{t:divker_sde_noh0} become
\begin{equation*}\begin{split}
  \nabla \sigma(x) = - \sigma x, \quad
  \nabla \sigma \nabla \sigma^T \nu = \sigma^2 \cdot x^T\nu \cdot x, \quad
  \Delta \sigma = \sigma \left( |x|^2 - M \right) , \quad
  \\
  \nabla^2 \sigma = \sigma (x x^T - I), \quad
  \nabla^2 \sigma \nabla \sigma = \sigma^2 (-x x^T x  + x), \quad
  \\
  \div F = - 40 - 0.02 \sum x^i, \quad
  \nabla \div F = - 0.02 [1, \ldots, 1].
\end{split}\end{equation*}
Note that the initial distribution is singular, so $\nabla\log h_0$ is not even defined; fortunately, \Cref{t:divker_sde_noh0} does not need it.
Set $T=0.3$ so that the expression in \Cref{t:divker_sde_noh0} is not too large.

To verify our new formula, we employ an observable and a constant $v_T$,
\[
  \Phi(x) := \frac 1M \left( x^1+\cdots x^M \right),
  \quad \textnormal{} \quad 
  v_T(x) \equiv [1,\ldots, 1].
\]
Our goal is to compute the linear response
\begin{equation*}\begin{split}
  \delta \E { \Phi(x_T - \gamma v_T)} 
  :=\delta \int \Phi(x_T - \gamma v_T) h_T(x_T) dx_T,
  \quad \textnormal{where} \quad 
  \delta(\cdot):= \pp{}{\gamma} |_{\gamma=0}.
\end{split}\end{equation*}
There are two ways to compute it; the first is 
\begin{equation*}\begin{split}
  \delta \int \Phi(x_T - \gamma v_T) h_T(x_T)  dx_T
  = 
  - \int \nabla_{v_T} \Phi(x_T) h_T(x_T) dx_T
  = -1
  .
\end{split}\end{equation*}
The second is to compute by averaging our new formula.
First, change the dummy variable $x_T$ to $y_T + \gamma v_T$, whose Jacobian matrix is $ \pp y x = I -\gamma \pp vx $, so 
\begin{equation*}\begin{split}
  \int \Phi(x_T - \gamma v_T) h_T(x_T)  dx_T
  = 
  \int \Phi(y_T) h_T(y_T + \gamma v_T) 
  \det\left(I - \gamma \pp vx \right)^{-1} dy_T.
\end{split}\end{equation*}
Using similar arguments as in the proof of \Cref{t:div},
and note that $\div v_T=0$,
\begin{equation*}\begin{split}
  \delta \E { \Phi(x_T - \gamma v_T)} 
  = 
  \E{
  \Phi(x_T) \left(\nabla_{v_T} \log h_T (x_T) 
  + \div v_T(x_T)
  \right)
  }
  = 
  \E{
  \Phi(x_T) \nu_T\cdot v_T
  } .
\end{split}\end{equation*}
Here $\nabla \log h_T$ is computed using our new formulas.
Note that here $x_T$ is integrated, so we no longer require many samples with the same $x_T$; we can have many sample paths of the SDE, then compute one $\nu_T$ on each of the paths, then average these $\nu_T$'s.

For SDEs, we use the Euler integration scheme with $\Dt=0.002$.
A typical orbit is in \Cref{f:lorenz}.
Then we use \Cref{t:divker_sde_noh0} to compute $\nu_T$ at $T=0.3$.
Then we compute the deviation 
\begin{equation*}\begin{split}
  \Phi(x_T) \nu_T \cdot v_T + 1.
\end{split}\end{equation*}
By the above analysis, the deviation should have zero average.
The histogram of the computed deviations is shown in \Cref{f:lorenz}; the average of the computed deviations is close to zero, indicating that the score computed by the divergence-kernel method is correct.

\begin{figure}[ht] \centering
  \includegraphics[width=0.45\textwidth]{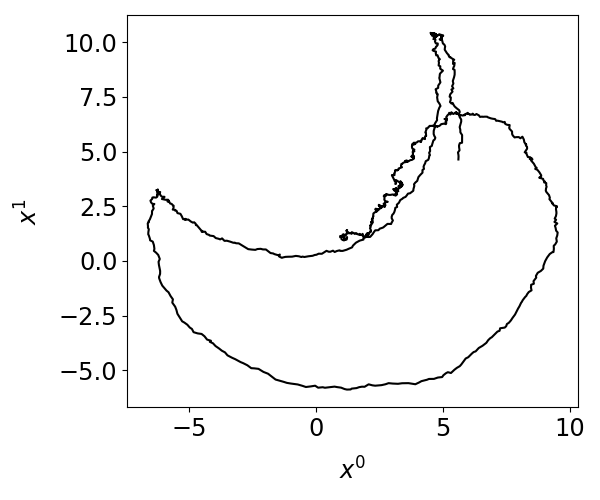} \hfill
  \includegraphics[width=0.45\textwidth]{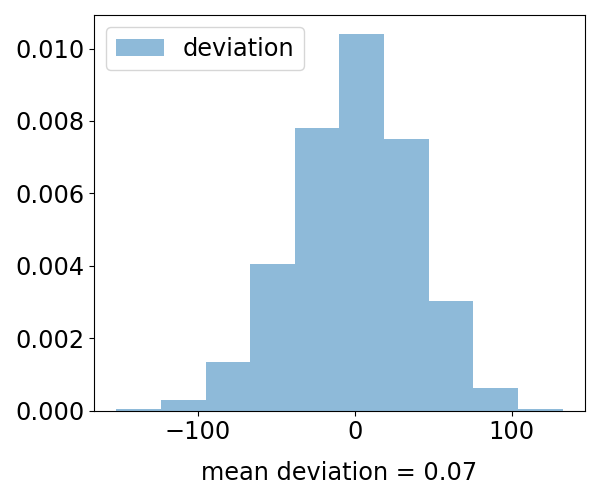}
  \caption{
  Left: plot of $x^0_t, x^1_t$ from a typical orbit of time length $3$. 
  Right: histogram of the deviations computed on $10000$ sample paths, whose theoretical expectation should be zero.
  }
  \label{f:lorenz}
\end{figure}



\section*{Data availability statement}
The code used in this paper is posted at \url{https://github.com/niangxiu/divker}.
There are no other associated data.



\bibliographystyle{abbrv}
{\footnotesize\bibliography{library}}

\end{document}